\documentclass[11pt]{article}
\usepackage{eurosym}
\usepackage{amsfonts}
\usepackage{amssymb}
\usepackage{amsmath}
\usepackage{latexsym}
\usepackage{a4wide}
\usepackage{graphicx}

\newtheorem{theorem}{Theorem}

\newtheorem{corollary}[theorem]{Corollary}

\newtheorem{example}[theorem]{Example}

\newtheorem{lemma}[theorem]{Lemma}

\newtheorem{proposition}[theorem]{Proposition}
\newtheorem{remark}[theorem]{Remark}

\newenvironment{proof}[1][Proof]{\noindent\textsc{#1.}\,\  }{\ \textsc{q.e.d.}\medskip}
\renewcommand{\sl}{\it}
\renewcommand{\iff}{\leftrightarrow}

\renewcommand{\textsl}{\textit}

\renewcommand{\geq}{\geqslant}
\input{tcilatex}
\renewcommand{\phi}{\varphi}
\providecommand{\psi}{\psi}

\begin{document}

\title{Are There Enough Injective Sets?}
\author{Peter Aczel, Benno van den Berg, Johan Granstr\"{o}m, and Peter
Schuster}
\maketitle

\begin{abstract}
The axiom of choice ensures precisely that, in ZFC, every set is
projective:\ that is, a projective object in the category of sets. In
constructive ZF (CZF) the existence of enough projective sets has been
discussed as an additional axiom taken from the interpretation of CZF in
Martin-L\"{o}f's intuitionistic type theory. On the other hand, every
non--empty set is injective in classical ZF, which argument fails to work in
CZF. The aim of this paper is to shed some light on the problem whether
there are (enough) injective sets in CZF.

We show that no two element set is injective unless the law of excluded
middle is admitted for negated formulas, and that the axiom of power set is
required for proving that \textquotedblleft there are strongly enough
injective sets\textquotedblright . The latter notion is abstracted from the
singleton embedding into the power set, which ensures enough injectives both
in every topos and in IZF. We further show that it is consistent with CZF to
assume that the only injective sets are the singletons. In particular,
assuming the consistency of CZF one cannot prove in CZF that there are
enough injective sets. As a complement we revisit the duality between
injective and projective sets from the point of view of intuitionistic type
theory.
\end{abstract}

\section{Introduction}

What are injective objects good for? In abelian categories, injective
resolutions are used to define and compute the right derived functors of a
left exact covariant functor. A famous instance is the cohomology theory of
sheaves, which has contributed to the settling of Fermat's conjecture \cite%
{McLarty}. \textquotedblleft A standard method is:\ Take a resolution, apply
a covariant functor $T$ \ldots , take the [co]homology of the resulting
[co]complex. This gives a connected sequence of functors, called the derived
functors of $T$.\textquotedblright\ \cite[p.~389]{MacLane}~

To have injective resolutions one needs to have what is called enough
injective objects:\ that is, every object can be embedded into an injective
object. For the prime example of an abelian category, the category of
abelian groups,\textbf{\ }\textquotedblleft \ldots\ the usual proof \ldots\
consists of two major steps. First every abelian group is a subgroup of a
divisible one. Second, all divisible abelian groups are injective. \ldots\
the first step can be carried out in \textbf{ZFA}. \ldots\ the second step
is \ldots\ equivalent to the axiom of choice.\textquotedblright\ \cite[p.~34]%
{Blass}\footnote{\textbf{ZFA }denotes \textbf{ZF} with atoms.}

In \textbf{ZFC }the axiom of choice ensures precisely that every set is
projective:\ that is, a projective object in the category of sets. In
constructive \textbf{ZF} (\textbf{CZF}) \cite{Aczel78,Aczel82,acra}, the
framework of the present note, the existence of enough projective sets has
been discussed as an additional axiom. This presentation axiom is taken from
the interpretation \cite{Aczel78} of \textbf{CZF }in Martin-L\"{o}f's
intuitionistic type theory \textbf{ITT }\cite{ITT}. The dual notion of an
injective set is trivial in \textbf{ZF}, where a set is an injective\textbf{%
\ }object of the category of sets precisely when it is non--empty. In
particular, there are enough injective sets in \textbf{ZF}, which argument
will turn out to fail in \textbf{CZF}.

More precisely, we will show that no two element set is injective unless the
law of excluded middle is admitted for negated formulas; and that the axiom
of power set is required for proving that \textquotedblleft there are
strongly enough injective sets\textquotedblright . The latter notion is
abstracted from the singleton embedding into the power set, which ensures
enough injective objects in every topos \cite[IV.10, Corollary 3]{MM} and
likewise in the intuitionistic \textbf{ZF} (\textbf{IZF}) from \cite%
{Friedman}. We further give an argument that it is consistent with \textbf{%
CZF} to assume that the only injective sets are the singletons. In
particular, assuming the consistency of \textbf{CZF }it cannot be proved in 
\textbf{CZF }that there are enough injective sets. As a complement we
revisit the duality between injective and projective sets from \textbf{ITT}%
's perspective.

\section{Preliminaries}

\subsection{Constructive Set Theories}

The framework of this paper is the constructive Zermelo--Fraenkel set theory
(\textbf{CZF}) begun with \cite{Aczel78}. While \textbf{CZF} is formulated
in the same language as \textbf{ZF}, it is based on intuitionistic rather
than classical logic; from \textbf{CZF} one arrives at \textbf{ZF} by adding
the law of excluded middle. Moreover, the axiom of power set does not belong
to \textbf{CZF}, for this theory is also intended to be predicative (in a
generalised sense).

In most of the paper we can work in \textbf{CZF}'s fragment \textbf{CZF}$%
_{0} $ from \cite{acra} which has the following set--theoretic axioms and
axiom schemes: extensionality, pairing, union, replacement, restricted
separation, strong infinity, and mathematical induction. All these
principles will be recalled in the appendix. Apart from the different choice
of the underlying logic, the basic set theory from \cite[p.~36]{Devlin} is a
fragment of \textbf{ZF} that plays a role roughly analogous to the one
played by \textbf{CZF}$_{0}$ within \textbf{CZF}. In addition to \textbf{CZF}%
$_{0}$, we sometimes need to assume the principle of relativised dependent
choices (RDC) from \cite{Aczel82}.

The axiom scheme of restricted separation only allows separation by
restricted formulas, in which every quantifier must be bounded by a set. If
restricted separation and replacement are strengthened to full separation
and collection, respectively, and the axiom of power set is allowed, then
one obtains the intuitionistic Zermelo--Fraenkel set theory (\textbf{IZF})
initiated in \cite{Friedman}. Albeit of an impredicative nature, this 
\textbf{IZF} is still based on intuitionistic logic.

A set $S$ is \emph{inhabited} if it has an element. If a set theory---such
as \textbf{CZF }and \textbf{IZF}---is based on intuitionistic logic,
\textquotedblleft $S$ is inhabited\textquotedblright\ has to be
distinguished from \textquotedblleft $S$ is non--empty\textquotedblright :
the latter is the double negation of the former. For a similar reason we
need to recall that a set $X$ is \emph{detachable} from a superset $Y$ if
membership to $X$ is a decidable predicate on $Y$: i.e., 
\begin{equation*}
\forall y\in Y\,\left( y\in X\vee y\notin X\right) \,.
\end{equation*}%
A set $E$ is \emph{discrete }if any singleton subset is detachable from $E$
or, equivalently, if equality is a decidable relation on $E$: that is, 
\begin{equation*}
\forall u,v\in E\,\left( u=v\vee u\not=v\,\right) \,.
\end{equation*}%
A set $E$ is \emph{finitely enumerable} if there is $n\in \mathbb{N}$ and a
surjective map from $n$ to $E$. A set $E$ is \emph{finite }if $E$ is in
bijection to some $n\in \mathbb{N}$, which uniquely determined $n$ is the 
\emph{cardinality} of $E$. A set $E$ is finite if and only if it is finitely
enumerable and discrete.

For more details we refer to \cite{acra,MRR}. We assume that every map
between sets is a set.

\subsection{Injective Maps and Sets}

By an \emph{embedding} we understand an injective map. Given a subset $U$ of
a set $V$, the \emph{inclusion} $U\hookrightarrow V$ is an embedding. Every
embedding is the composition of a bijection followed by an inclusion. We say
that

\begin{itemize}
\item[---] a set $E$ is \emph{injective} if every map with codomain $E$ can
be extended to any superset of its domain;

\item[---] \emph{there are enough injective sets} if every set is the domain
of an embedding whose codomain is an injective set.
\end{itemize}

An object $E$ of a category is called injective if, given a monomorphism $%
X\rightarrow Y$, every morphism from $X$ to $E$ can be extended to a
morphism from $Y\ $to $E$. Since the (mono)morphisms in the category of sets
are precisely the (injective)\ maps, a set is injective if and only if it is
an injective object in the category of sets; whence there are enough
injective sets if and only if the category of sets has enough injective
objects. In particular, the notion of an injective set is a \emph{structural
notion}:\ that is, given any bijection between two sets, if one of them is
injective, then so is the other.

A set $X$ is said to be a \emph{retract }of a superset $Y$ if the inclusion $%
i:X\hookrightarrow Y$ has a left inverse: i.e., a map $r:Y\rightarrow X$
with $ri=\mathrm{id}_{X}$. Clearly, an inhabited set $X$ is a retract of any
superset $Y$ from which $X$ is detachable, in which case every map from $X$
to a set $E$ can be extended to $Y$. In this sense, an inhabited set $E$ is
injective with respect to inclusions $X\hookrightarrow Y$ of detachable
subsets. Apart from this, the following fragments of \textbf{ZF}'s property
\textquotedblleft the injective sets are precisely the non--empty
sets\textquotedblright\ remain valid in \textbf{CZF}.

\begin{lemma}
\label{singletoninhabited}Every singleton set is injective, and every
injective set is inhabited.
\end{lemma}

\begin{proof}
Clearly every map to a singleton can be extended to any superset of its
domain. If $E$ is an injective set, then the one and only map from $0$ to $E$
can be extended to a map from $1$ to $E$; whence $E$ is inhabited.
\end{proof}

The counterpart of the subsequent lemma holds in any category.

\begin{lemma}
\label{retract}An injective set is a retract of any given superset, and
every retract of an injective set is injective.
\end{lemma}

Let $X$ be a set. An $E$\emph{--partition }of $X$ is a family, indexed by $E$%
, of pairwise disjoint subsets whose union equals $X$. An $E$--partition $%
\left( Y_{e}\right) _{e\in E}$ of a set $Y$ with $X\subseteq Y$ is an \emph{%
extension} of an $E$--partition $\left( X_{e}\right) _{e\in E}$ of $X$ if $%
X_{e}\subseteq Y_{e}$ for all $e\in E$.

The next lemma is clear from the obvious one--to--one correspondence between
maps from $X$ to $E$ and $E$--partitions of $X$, which in the context of 
\textbf{CZF }has proved useful before \cite{Expon,Shrink,Complet}.

\begin{lemma}
\label{part}A set $E$ is injective precisely when, for any pair of sets $%
X\subseteq Y$, every $E$--partition of $X$ can be extended to an $E$%
--partition of $Y$.
\end{lemma}

Finite products of injective sets are injective. We next give a partial dual.

\begin{lemma}
\label{partcor}Let $I$ be a finite set, and $\left( E_{i}\right) _{i\in I}$
an $I$--partition of a set $E$. If $I$ is injective, and $E_{i}$ is an
injective set for every $i\in I $, then $E$ is injective.
\end{lemma}

\begin{proof}
Let $X\subseteq Y$ be sets, and $f:X\rightarrow E$ a map. We use Lemma \ref%
{part}. Since $I$ is injective, the $I$--partition $\left( X_{i}\right)
_{i\in I}$ of $X$ defined by $X_{i}=f^{-1}\left( E_{i}\right) $ can be
extended to an $I$--partition $\left( Y_{i}\right) _{i\in I}$ of $Y$. For
each $i\in I$ the map $X_{i}\rightarrow E_{i}$ induced by $f$ can be
extended to a map $g_{i}:Y_{i}\rightarrow E_{i}$, because $E_{i}$ is
injective. Now the $g_{i}$ with $i\in I$ define a map $g:Y\rightarrow E$
that extends $f$.
\end{proof}

\noindent In the case of an \emph{arbitrary} index set $I$ one would need to invoke the axiom of choice to
choose for every $i\in I$ an extension $g_{i}$ of $f_{i}$. In \textbf{CZF},
however, choice functions {on \emph{finite }}sets can be defined as usual. 

\section{Injective Sets and Excluded Middle}

For every subset $p$ of $1$ it is plain that 
\begin{equation*}
\begin{array}{ccc}
p=0 & \iff & 0\notin p\,, \\ 
p=1 & \iff & 0\in p\,.%
\end{array}%
\end{equation*}%
If $\varphi $ is a restricted formula, then 
\begin{equation*}
p_{\varphi }=\{0:\varphi \}=\{x\in 1:\varphi \}
\end{equation*}%
is a subset of $1$, for which%
\begin{equation*}
\begin{array}{ccccc}
\varphi & \iff & 0\in p_{\varphi } & \iff & p_{\varphi }=1\,, \\ 
\lnot \varphi & \iff & 0\notin p_{\varphi } & \iff & p_{\varphi }=0\,.%
\end{array}%
\end{equation*}%
The \emph{Law of Restricted (Weak) Excluded Middle, }for short $\mathrm{%
R(W)EM}$, says that $\varphi \vee \lnot \varphi $ (respectively, $\lnot
\varphi \vee \lnot \lnot \varphi $) holds for every restricted formula $%
\varphi $.

\begin{proposition}
Each of the following is equivalent to $\mathrm{REM}$:

\begin{enumerate}
\item Every inhabited set is injective.

\item Every inhabited set is a retract of any given superset.
\end{enumerate}
\end{proposition}

\begin{proof}
To see that REM implies the first item, let $E$ be an inhabited set, and
take any $e\in E$. To show that $E$ is injective, let $Y$ be a superset of
the set $X$. With REM\ one can decide whether any given element of $Y$
belongs to $X$. Hence every $f:X\rightarrow E$ can be extended to $%
g:Y\rightarrow E$ by setting $g(y)=f(y)$ if $y\in Y$ and $g(y)=e$ if $%
y\notin X$.

By Lemma \ref{retract} the first item implies the second. To verify that REM
follows from the second item, let $\varphi $ be a restricted formula and
consider $X=$ $p_{\varphi }\cup \{1\}$ as a subset of $2$. If there is a map 
$r:2\rightarrow X$ whose restriction to $X$ is the identity on $X$, then $%
r(1)=1$ but%
\begin{equation*}
r(0)=0\iff 0\in X\iff 0\in p_{\varphi }\iff \varphi \,.
\end{equation*}%
In particular, we can decide $\varphi $, for by $X\subseteq 2$ we can decide 
$r(0)=0$.
\end{proof}

\begin{lemma}
\label{two}If there is an injective set $E$ that has an element $x_{0}$ such
that%
\begin{equation*}
\forall x\in E\,\left( x=x_{0}\vee x\neq x_{0}\right)
\end{equation*}%
and for which there is $x_{1}\in E$ with $x_{0}\neq x_{1}$, then $\mathrm{%
RWEM}$ holds.
\end{lemma}

\begin{proof}
Suppose that $E$ is a set as in the hypothesis. We may assume that $x_{0}=0$
and $x_{1}=1$. Let $\varphi $ be a restricted formula, and set $p=p_{\varphi
}$. Since $E$ is injective, there is a map $g:2\cup \{p\}\rightarrow E$
which extends the inclusion $2\hookrightarrow E$. In particular, $g\left(
0\right) =0$ and $g\left( 1\right) =1$. We now consider the element $x=$ $%
g\left( p\right) $ of $E$, for which 
\begin{equation*}
p=0\iff x=0\,.
\end{equation*}%
(If $p=0$, then clearly $x=0$, because $g\left( 0\right) =0$. Conversely, if 
$x=0$, then $0\notin p$, i.e.~$p=0$, for if $0\in p$, then $p=1$ and thus $%
x=1$, because $g\left( 1\right) =1$.) By assumption, $x=0\vee x\neq 0$ or,
equivalently, $p=0\vee p\neq 0$, which is to say that $\lnot \varphi \vee
\lnot \lnot \varphi \,$.
\end{proof}

\noindent In other words, \textrm{RWEM} follows from the statement that
every singleton $S$ can be embedded into an injective set $E$ such that the
image of $S$ is detachable from $E$ and has inhabited complement.

\begin{corollary}
\label{almost}If $\lnot \mathrm{RWEM}$, then for any injective set $E$ and
every $x_{0}\in E$ such that%
\begin{equation*}
\forall x\in E\,\left( x=x_{0}\vee x\neq x_{0}\right)
\end{equation*}%
there is no $x_{1}\in E$ with $x_{1}\neq x_{0}\,$.
\end{corollary}

\noindent In other words, $\lnot \mathrm{RWEM}$ implies that any given
injective set is\emph{\ almost a singleton set}: that is, it has a singleton
subset (Lemma \ref{singletoninhabited}), and every detachable singleton
subset has empty complement.

\begin{theorem}
\label{weak}Each of the following is equivalent to $\mathrm{RWEM}$:

\begin{enumerate}
\item The set $2$ is injective.

\item Every disjoint union of $n\geq 2$ injective sets is injective.

\item Every finite set of cardinality $\geq 2$ is injective.

\item There is a discrete injective set $E$ with $\geq 2$ elements.
\end{enumerate}
\end{theorem}

\begin{proof}
Using Lemma \ref{part} we first show that in the presence of RWEM the set $2$
is injective. Let $X\subseteq Y$ be sets. If $(X_{0},X_{1})$ is a $2$%
--partition of $X$, then the subsets%
\begin{equation*}
Y_{0}=\{y\in Y:\lnot \left( y\in X_{1}\right) \}\,,\quad Y_{1}=\{y\in
Y:\lnot \lnot \left( y\in X_{1}\right) \}
\end{equation*}%
of $Y$ form, by RWEM, a $2$--partition of $Y$.

To deduce the second item from the first, let $2$ be injective. By induction
on $n$ it suffices to prove that if $E_{0}$, $E_{1}$ are disjoint injective
sets, then $E_{0}\cup E_{1}$ is injective, which is a special case of Lemma %
\ref{partcor}. It is plain that item 2 implies item 3; that item 1 follows
from item 3; and that item 1 implies item 4. Finally, RWEM\ follows from
item 4 by Lemma \ref{two}.
\end{proof}

By Lemma \ref{retract}, if $\mathbb{N}$ is injective, then so is every $n\in 
\mathbb{N}$; and if $n\in \mathbb{N}$ is injective, then so is every $m\in n$.

According to \cite{BM11} the \emph{General Uniformity
Principle}%
\begin{equation*}
\forall x\exists y\in a\varphi \left( x,y\right) \rightarrow \exists y\in
a\forall x\varphi \left( x,y\right)
\end{equation*}%
is consistent with \textbf{CZF}.

\begin{theorem}
\label{gup} The general uniformity principle implies that the only injective
sets are the singletons.
\end{theorem}

\begin{proof}
In view of Lemma \ref{singletoninhabited} it remains to show that if $E$ is
an injective set, and $u,v\in E$, then $u=v$. (A set of this kind will later
be called a subsingleton, see Proposition \ref{subsin} below.) Let $x$ be an
arbitrary set, and let $\sim $ be the equivalence relation on $3=\left\{
0,1,2\right\} $ defined by%
\begin{equation*}
0\sim 1\iff \bot \,;\quad 0\sim 2\iff 0\in x\,;\quad 1\sim 2\iff 0\notin x\,.
\end{equation*}%
The inclusion $2\rightarrow 3$ followed by the projection $3\rightarrow
3/\!\sim $ gives an embedding $2\rightarrow 3/\!\sim $; whence by the
assumption that $E$ is an injective set the map $f:2\rightarrow E$ defined
by $f\left( 0\right) =u$ and $f\left( 1\right) =v$ can be extended to a map $%
g:3/\!\sim \rightarrow E$. Now there is $y\in E$ (for example, the image
under $g$ of the equivalence class of $2\in 3$) for which $\varphi \left(
x,y\right)$ holds where 
\begin{equation*}
\varphi \left( x,y\right) \equiv \left( 0\in x\rightarrow y=u\right) \wedge
\left( 0\notin x\rightarrow y=v\right) \,.
\end{equation*}%
In all, we have $\forall x\exists y\in E\varphi \left( x,y\right) $; whence
by the general uniformity principle there is $y\in E$ such that $\varphi
\left( x,y\right) $ for every $x$. By using this for any such $y\in E $, and
for $x=1$ and $x=0$, we arrive at $y=u$ and $y=v$, respectively; whence $u=v 
$.
\end{proof}

\begin{corollary}
\label{con} It is consistent with $\mathbf{CZF}$\textbf{\ }to assume that
the only injective sets are the singletons. In particular, under the
assumption that $\mathbf{CZF}$\textbf{\ }be consistent it cannot be proved
in $\mathbf{CZF}$\textbf{\ }that there are enough injective sets.
\end{corollary}

\noindent Does \textquotedblleft the only injective sets are the
singletons\textquotedblright\ have any other interesting consequences?

\begin{theorem}
Assume $\mathrm{RDC}$. If $2$ is injective, then $\mathbb{N}$ is injective.
\end{theorem}

\begin{proof}
We again use Lemma \ref{part}. Let $X\subseteq Y$ be sets, and $\left(
X_{n}\right) _{n\in \mathbb{N}}$ an $\mathbb{N}$--partition of $X$. Set $%
X_{n}^{\prime }=\bigcup_{m>n}X_{m}$ for each $n\in \mathbb{N}$. Note that

\begin{itemize}
\item[---] $(X_{0},X_{0}^{\prime })$ is a $2$--partition of $X$;

\item[---] $(X_{n+1},X_{n+1}^{\prime })$ is a $2$--partition of $%
X_{n}^{\prime }$ for every $n\in \mathbb{N}$.
\end{itemize}

\noindent%
Now suppose that $2$ is injective. We notice the following immediate
consequences:

\begin{itemize}
\item[---] There is a $2$--partition $\left( Y_{0},Y_{0}^{\prime }\right) $
of $Y$ that extends $\left( X_{0},X_{0}^{\prime }\right) $.

\item[---] For each $n\in \mathbb{N}$, given any superset $Y_{n}^{\prime }$
of $X_{n}^{\prime }$ there is a $2$--partition $(Y_{n+1},Y_{n+1}^{\prime })$
of $Y_{n}^{\prime }$ that extends $(X_{n+1},X_{n+1}^{\prime })$.
\end{itemize}

\noindent%
With RDC at hand we can choose a sequence of pairs $(Y_{n},Y_{n}^{\prime
})_{n\in \mathbb{N}}$ of subsets of $Y$ such that

\begin{itemize}
\item[---] $(Y_{0},Y_{0}^{\prime })$ is a $2$--partition of $Y$ that extends 
$(X_{0},X_{0}^{\prime })$;

\item[---] $(Y_{n+1},Y_{n+1}^{\prime })$ is a $2$--partition of $%
Y_{n}^{\prime }$ that extends $(X_{n+1},X_{n+1}^{\prime })$ for every $n\in 
\mathbb{N}$.
\end{itemize}

\noindent%
In particular, we have an $\mathbb{N}$--partition $\left( Y_{n}\right)
_{n\in \mathbb{N}}$ of $Y$ that extends $\left( X_{n}\right) _{n\in \mathbb{N%
}}$.
\end{proof}

With Theorem \ref{weak} we have the following.

\begin{corollary}
If $\mathrm{RDC}$ is assumed, then $\mathbb{N}$ is injective if and only if $%
\mathrm{RWEM}$\textrm{\ }holds.
\end{corollary}

\section{Strong Injectivity and Power Set}

It is well known that in \textbf{IZF }there still are enough injective sets,
because every power class is an injective set. An analysis of the proof has
prompted the following considerations, for which we need to assume that many
a set consists of sets, as is the case for all sets in \textbf{CZF}.

As usual we write $\mathcal{P}\left( E\right) $ for the power class of a set 
$E$. The axiom of power set is equivalent, in \textbf{CZF}$_{0}$ plus the
axiom of exponentiation, to the statement that $\mathcal{P}\left( 1\right) $
is a set: by exponentiation, if $\mathcal{P}\left( 1\right) $ is a set, then
so is $\mathcal{P}\left( Z\right) \cong \mathcal{P}\left( 1\right) ^{Z}$ for
every set $Z$. Needless to say, $\mathcal{P}\left( 1\right) $ is a set
already if the power class of an arbitrary singleton is a set.

\begin{lemma}
\label{sub}Let $S$ and $s$ be sets. If $S\subseteq \{s\}$, then $\bigcup
S\subseteq s$, and $\bigcup S=s$ precisely when $S=\{s\}$.
\end{lemma}

We say that a set $S$ is a \emph{subsingleton} if $S$ is a subset of $\{s\}$
for some set $s$.

\begin{lemma}
\label{singlemaps}Every set can be mapped \emph{onto} a subsingleton set.
\end{lemma}

\begin{proof}
Every set is the domain of a mapping whose codomain is any singleton set.
\end{proof}

We write $\mathcal{P}_{1}\left( E\right) $ for the class of subsingleton
subsets of $E$. While $\mathcal{P}_{1}\left( 0\right) =\mathcal{P}\left(
0\right) $ is a set, $\mathcal{P}_{1}\left( 1\right) =\mathcal{P}\left(
1\right) $ and thus $\mathcal{P}_{1}\left( E\right) $ for any inhabited set $%
E$ are proper classes unless the axiom of power set is assumed.

\begin{proposition}
\label{subsin}The following are equivalent for any set $S$:

\begin{enumerate}
\item $S$ is a subsingleton.

\item $x=y$ for all $x,y\in S$.

\item $S$ is a subset of $\left\{ \bigcup S\right\} $.

\item Every mapping with domain $S$ is an embedding.
\end{enumerate}
\end{proposition}

\noindent%
In particular, $\mathcal{P}_{1}\left( E\right) $ is closed under forming
subsets and under taking images: every subset of a subsingleton is a
subsingleton; and if $f:X\rightarrow E$ is a map between sets, then $f\left(
S\right) $ is a subsingleton for any subsingleton $S$ with $S\subseteq X$.

We say that a set $E$ of sets\footnote{%
That $E$ is a set \emph{of sets} is required whenever atoms are allowed in
the set theory under consideration.} is $\mathcal{P}_{1}$--\emph{complete}
if $\bigcup S\in E$ for every $S\in \mathcal{P}_{1}\left( E\right) $. The
prime examples are the power sets: if $Z$ is a set such that $\mathcal{P}%
\left( Z\right) $ is a set, then $\mathcal{P}\left( Z\right) $ is $\mathcal{P%
}_{1}$--complete. In particular, $1=\mathcal{P}\left( 0\right) $ is a $%
\mathcal{P}_{1}$--complete set.

\begin{corollary}
A set $E$ of sets is $\mathcal{P}_{1}$--complete if and only if $\bigcup
f\left( S\right) \in E$ whenever $X$ is a set, $f:X\rightarrow E$ a map, and 
$S\in \mathcal{P}_{1}\left( X\right) $.{}
\end{corollary}

\begin{proposition}
\label{hat}Let $E$ be a set of sets. If $f:X\rightarrow E$ is a map between
sets, $Y$ a superset of $X$, and $y\in Y$, then the set%
\begin{equation*}
\widehat{f}\left( y\right) =\bigcup \left\{ f\left( x\right) \,:x\in
\{y\}\cap X\right\}
\end{equation*}%
possesses the following properties:

\begin{enumerate}
\item If $y\notin X$, then $\widehat{f}\left( y\right) =0$;

\item If $y\in X$, then $\widehat{f}\left( y\right) =f\left( y\right) $;

\item If $E$ is $\mathcal{P}_{1}$--complete, then $\widehat{f}\left(
y\right) \in E$.

\item If $E$ is $\mathcal{P}_{1}$--complete, then $\widehat{f}:Y\rightarrow
E $ extends $f$.

\item If $h:Y\rightarrow E$ extends $f$, then $\widehat{f}\left( y\right)
\subseteq h\left( y\right) $.
\end{enumerate}
\end{proposition}

\begin{proof}
Part 1 is obvious. Part 2. If $y\in X$, then $\{y\}\cap X=\{y\}$. Part 3.
For every set $y$ the subset $\{y\}\cap X$ of $X$ is a subsingleton. Part 4
is an immediate consequence of parts 2 and 3. Part 5. Let $y\in Y$. If $z\in 
\widehat{f}\left( y\right) $, then $z\in f\left( x\right) $ for some $x\in
\{y\}\cap X$, for which $f\left( x\right) =h(y)$.
\end{proof}

\noindent%
The following is best seen in the light of Lemma \ref{singletoninhabited}.

\begin{corollary}
\label{strongweak}If $E$ is a $\mathcal{P}_{1}$--complete set, then $0\in E$%
, and $E$ is injective.
\end{corollary}

\begin{proof}
Let $E$ be a $\mathcal{P}_{1}$--complete set. To see that $0\in E$, apply
parts 1 and 3 of Proposition \ref{hat} to $X=0$, $Y=1$, and $y=0$; as for $E$
being injective, use parts 3 and 4.
\end{proof}

\noindent%
In particular, $1$ is the only $\mathcal{P}_{1}$--complete subsingleton.
Hence there are plenty of injective sets which are not $\mathcal{P}_{1}$%
--complete (e.g., the singletons different from $1$); and $\mathcal{P}_{1}$%
--completeness is---unlike injectivity---not a structural notion. With this
warning we say that

\begin{itemize}
\item[---] an embedding $i:U\rightarrow V$ of sets is a \emph{strong
embedding }if $i(x)$ is a singleton for every $x\in U$;

\item[---] \emph{there are strongly enough injective sets} if each set is
the domain of a strong embedding whose codomain is a $\mathcal{P}_{1}$%
--complete set.
\end{itemize}

\noindent%
Every strong embedding $f:X\rightarrow Y$ is an embedding; whence if there
are strongly enough injective sets, then there are enough injective sets.

The following prime example of a strong embedding follows the proof that
every topos has enough injective objects \cite[IV.10, Corollary 3]{MM}.

\begin{example}
\label{bsp}Let $Z$ be a set. If $\mathcal{P}\left( Z\right) $ is a set, then
the \emph{singleton embedding}%
\begin{equation*}
Z\hookrightarrow \mathcal{P}\left( Z\right) \,,\quad z\mapsto \{z\}
\end{equation*}%
is a strong embedding, and $\mathcal{P}\left( Z\right) $ is $\mathcal{P}_{1}$%
--complete.
\end{example}

\noindent%
In particular, there are (strongly) enough injective sets in \textbf{IZF}.

\begin{lemma}
\label{sgl}If there is a $\mathcal{P}_{1}$--complete set some element of
which is inhabited, then $\mathcal{P}\left( 1\right) $ is a set.
\end{lemma}

\begin{proof}
Let $E$ be a $\mathcal{P}_{1}$--complete set, and $x\in Y\in E$. For each $%
p\in \mathcal{P}\left( 1\right) $ set%
\begin{equation*}
\,Y_{p}=\{y\in Y:0\in p\},\quad S_{p}=\{Z\in E:Z=Y\wedge 0\in p\}\,.
\end{equation*}%
The following assertions are readily seen to be equivalent: 
\begin{equation*}
0\in p;~Y_{p}=Y;~x\in Y_{p};~Y_{p}~\text{is inhabited};~S_{p}=\left\{
Y\right\} ;~S_{p}~\text{is inhabited}.
\end{equation*}%
Since $Y_{p}=\bigcup S_{p}$ and $S_{p}\in \mathcal{P}_{1}\left( E\right) $,
we have $Y_{p}\in E$. In other words,%
\begin{equation*}
F=\{Y_{p}:p\in \mathcal{P}\left( 1\right) \}
\end{equation*}%
is a subclass of the set $E$. Moreover, 
\begin{equation}
F=\{Z\in E:Z\subseteq Y\wedge \forall y\in Y\,(y\in Z\leftrightarrow x\in
Z)\}\,;  \label{eff}
\end{equation}%
whence $F$ is a set by restricted separation. (To see the part $\supseteq $
of (\ref{eff}), let $Z$ belong to the right-hand side. Set $p=p_{\varphi }$
where $\varphi $ stands for any of the following equivalent assertions:%
\begin{equation*}
Z=Y;~x\in Z;~Z~\text{is inhabited}.
\end{equation*}%
In particular, $0\in p$ is tantamount to any of these assertions; whence $%
Z=Y_{p}$ for this $p$.)

Also, for $p,q\in \mathcal{P}\left( 1\right) $, if $Y_{p}=Y_{q}$, then $p=q$%
. Hence for every $Z\in $ $F$ there is a uniquely determined $p\in \mathcal{P%
}\left( 1\right) $ with $Z=Y_{p}$; and $\mathcal{P}\left( 1\right) $ is a
set by replacement.
\end{proof}

\begin{theorem}
\label{power}With exponentiation, each of the following is equivalent to the
axiom of power set:

\begin{enumerate}
\item There are strongly enough injective sets.

\item Every inhabited set can be strongly embedded into a $\mathcal{P}_{1}$%
--complete set.

\item Every singleton set can be strongly embedded into a $\mathcal{P}_{1}$%
--complete set.

\item The singleton set $1$ can be strongly embedded into a $\mathcal{P}_{1}$%
--complete set.

\item There is a singleton set that can be strongly embedded into a $%
\mathcal{P}_{1}$--complete set.

\item There is an inhabited set that can be strongly embedded into a $%
\mathcal{P}_{1}$--complete set.

\item There is a $\mathcal{P}_{1}$--complete set some element of which is a
singleton.

\item There is a $\mathcal{P}_{1}$--complete set some element of which is
inhabited.
\end{enumerate}
\end{theorem}

\begin{proof}
Example \ref{bsp} says that the first item follows from the axiom of power
set, which in turn follows from the last item by way of Lemma \ref{sgl}.
\end{proof}

Although the singleton set $1$ is $\mathcal{P}_{1}$--complete, to give an
embedding of $1$ into a $\mathcal{P}_{1}$--complete set requires (Theorem %
\ref{power}) the axiom of power set. Note in this context that the identity
map on a set $S$ fails to be a strong embedding unless $S$ consists of
singletons, in which case the inclusion of $S$ into any superset is a strong
embedding. Now the simplest example of a set consisting of singletons is $%
\{1\}$. By Corollary \ref{strongweak}, however, there is no hope that $\{1\}$
be $\mathcal{P}_{1}$--complete, for it lacks the element $0$. The inclusion
of $\{1\}$ into $2$ is a strong embedding, and $2$ contains $0$. But $2$
cannot be $\mathcal{P}_{1}$--complete unless RWEM holds (Theorem \ref{weak},
Corollary \ref{strongweak}).

\begin{corollary}
If the axiom of power set is false, then $1$ is the only $\mathcal{P}_{1}$%
--complete set.
\end{corollary}

As we have noticed before, injectivity is a structural notion, whereas $%
\mathcal{P}_{1}$--completeness lacks this property in general. This defect
can be repaired by enriching the notion of $\mathcal{P}_{1}$--completeness.
We say that a pair $(X,E)$ is a \emph{strongly injective structure}, for
short \emph{sis}, if $X$ is a set and $E$ is a $\mathcal{P}_{1}$--complete
set of subsets of $X$ such that $\bigcup E=X$. As desired, sis is a
structural notion. If $X$ is a set such that $\mathcal{P}(X)$ is a set, then 
$(X,\mathcal{P}(X))$ is a sis. More generally, $(X,\mathcal{T})$ is a sis
whenever $\mathcal{T}$ is a set of subsets of a set $X$ on which $\mathcal{T}
$ is a topology.

\begin{remark}
If $(X,E)$ is a sis, then $0\in E$ and $E$ is injective.
\end{remark}

\begin{proposition}
The following are equivalent for any set $E$ of sets:

\begin{enumerate}
\item $E$ is $\mathcal{P}_{1}$--complete.

\item $(\bigcup E,E)$ is a sis.

\item There is a set $X$ such that $(X,E)$ is a sis.
\end{enumerate}
\end{proposition}

\section{Injectivity versus Projectivity}

In intuitionistic type theory \textbf{ITT} \cite{ITT} the notion of set is
intensional so that, to represent extensional mathematics in \textbf{ITT} it
is necessary to use the following representation of the sets and functions
of extensional mathematics:

\begin{enumerate}
\item[\textbf{R}] The sets of extensional mathematics are represented as
setoids; i.e. structures $(A,=_{A})$ consisting of a set $A$ (in the sense
of \textbf{ITT}) together with an equivalence relation $=_{A}$ on $A$. The
functions of extensional mathematics are represented as maps between setoids
that are required to respect the equivalence relations.
\end{enumerate}

\noindent Martin-L{\"{o}}f \cite{Martin-Loef} argues that in \textbf{ITT}
the axiom of choice is evident only for the intensional sets of \textbf{ITT}%
, but not the extensional setoids. He shows that, in \textbf{ITT}, the axiom
of choice for setoids is equivalent to the law of excluded middle.

When working in \textbf{ITT} we may refer to the sets of \textbf{ITT} as the 
\emph{intensional sets} in contrast to the setoids which we may call the 
\emph{extensional sets}.

Let $S$ be an intensional set, and $=_{S}$ an equivalence relation on $S$.
As any intensional set, $S$ is also equipped with the finest equivalence
relation $\mathrm{Id}_{S}$. We further have an extensional function $\varphi
:(S,\mathrm{Id}_{S})\rightarrow (S,=_{S})$, defined by $\varphi (s)=s$,
which respects equality just because $\mathrm{Id}_{S}$ is the finest
equivalence relation on $S$. It is clear that $\varphi $ is surjective. This
is why we have enough projective sets when whenever representation R is
adopted.

However, the \textquotedblleft natural dual\textquotedblright\ of this
construction does not imply that we have enough injective sets. Let $x\sim
_{S}y$ be a binary relation defined so that it is always true (and vacuously
an equivalence relation). Clearly, the extensional set $(S,\sim _{S})$ is a
subsingleton. A construction dual, in a sense, to the one of $\varphi $
above is the extensional function $\psi :(S,=_{S})\rightarrow (S,\sim _{S})$%
, again defined by $\psi (s)=s$. Now \emph{if }the intensional set $S$ is
inhabited, \emph{then} the extensional set $(S,\sim _{S})$ is a singleton,
and thus injective. However, there is no reason to believe that $(S,\sim
_{S})$ can be generally proved to be injective.

In all, with the representation R of extensional mathematics, the following
two statements are dual to each other.

\begin{enumerate}
\item[\textbf{A}] \emph{Every set is the codomain of a surjective function
from a projective set.}

\item[\textbf{B}] \emph{Every set is the domain of a surjective function to
a subsingleton set.}
\end{enumerate}

\noindent%
Both A and B\ are valid in \textbf{ITT }under representation R; more
precisely, A is a consequence of \textbf{ITT}'s axiom of choice for
intensional sets. In \textbf{CZF }one can prove B (Remark \ref{singlemaps}),
whereas A is nothing but the axiom of presentation taken \cite{Aczel82} from
the interpretation of \textbf{CZF }in \textbf{ITT}.

\section{Appendix: The Axioms of CZF$_{0}$}

The language of \textbf{CZF} is the first--order language of \textbf{ZF}
with the non--logical symbols $\in $ and $=$. The logical symbols are all
the intuitionistic operators $\perp $, $\wedge $, $\vee $, $\rightarrow $, $%
\exists $, and $\forall $;\ in particular, $\lnot \varphi $ is defined as $%
\varphi \rightarrow \perp $. A formula of \textbf{CZF} is \emph{restricted}
or is a\emph{\ $\Delta _{0}$--formula} if all quantifiers occurring in
it---if any---are bounded:\ that is, they are of the form $\exists x\in y$
or $\forall x\in y$, where $\exists x\in y\,\varphi $ and $\forall x\in
y\,\varphi $ stand for $\exists x\,(x\in y\;\wedge \;\varphi )$ and $\forall
x\,(x\in y\;\rightarrow \;\varphi )$, respectively. As usual, $x\subseteq y$
is a shorthand for $\forall z\,\in x\,\left( z\in y\right) $.

In addition to the usual axioms for intuitionistic first--order logic with
equality, the axioms of \textbf{CZF}$_{0}$ are the following seven
set--theoretic axioms and axiom schemes.

\begin{enumerate}
\item \textbf{Extensionality}\label{extensionality} 
\begin{equation*}
\forall a\,\forall b\;(a\subseteq b\wedge b\subseteq a\,\rightarrow \,a=b)\,.
\end{equation*}

\item \textbf{Pairing} 
\begin{equation*}
\forall a\,\forall b\,\exists x\,\forall y\;(y\in x\,\leftrightarrow
\,y=a\vee y=b)\,.
\end{equation*}

\item \textbf{Union} 
\begin{equation*}
\forall a\,\exists x\,\forall y\;(y\in x\,\leftrightarrow \,\exists z\in
a\,y\in z)\,.
\end{equation*}

\item \textbf{Replacement} \emph{For every formula }$\varphi \left(
x,y\right) $ \emph{in which }$b$ \emph{is not free},%
\begin{equation*}
\forall a\,\left( \forall x\in a\,\exists !y\;\varphi (x,y)\rightarrow
\,\exists b\;\forall y\,(y\in b\leftrightarrow \exists x\in a\,\varphi
(x,y))\right) \,.
\end{equation*}

\item \textbf{Restricted Separation}\label{boundedseparation}\thinspace\ 
\emph{For every }$\Delta _{0}$\emph{--formula }$\varphi \left( y\right) $ 
\emph{in which }$x$ \emph{is not free}, %
\begin{equation*}
\forall a\,\exists x\,\forall y\;(y\in x\,\leftrightarrow \,y\in a\,\wedge
\,\varphi (y))\,.
\end{equation*}

\item \textbf{Strong Infinity} 
\begin{equation*}
\exists x\,\left( \text{Ind}\left( x\right) \wedge \forall y\,\left( \text{%
Ind}\left( y\right) \rightarrow x\subseteq y\right) \right) \,,
\end{equation*}%
\emph{where the following abbreviations are used:}

Empty$\left( y\right) $ \emph{for }$\forall z\in y\,\perp \,;$

Succ$\left( y,z\right) $ \emph{for }$\forall u\,\left( u\in z\leftrightarrow
u\in y\vee u=y\right) \,;$

Ind$\left( x\right) $ \emph{for }$\exists y\in x\,$Empty$\left( y\right)
\wedge \forall y\in x\,\exists z\in x\,$Succ$\left( y,z\right) \,.$
\end{enumerate}

The axiom of strong infinity ensures the existence of a least inductive set,
which uniquely determined set is denoted by $\mathbb{N}$. The empty set $%
\emptyset $ can be defined e.g.~by restricted separation from $\mathbb{N}$;
and by extensionality $y^{\prime }\equiv y\cup \left\{ y\right\} $ is the
one and only successor of each $y\in \mathbb{N}$. Hence the elements of $%
\mathbb{N}$ are $0\equiv \emptyset $, $1\equiv \{0\}$, $2\equiv \{0,1\}$,
etc. Together with restricted separation, strong infinity allows for proofs
by induction of $\Delta _{0}$--formulas; to have this for arbitrary formulas
requires to adopt a further axiom scheme, which can be put as follows:

\begin{enumerate}
\item[7.] \textbf{Mathematical Induction }\emph{For every formula }$\phi
\left( x\right) $,%
\begin{equation*}
\phi \left( 0\right) \wedge \forall y\,\in \mathbb{N\,}\left( \phi \left(
y\right) \rightarrow \phi \left( y^{\prime }\right) \right) \,\rightarrow
\forall x\in \mathbb{N\,}\phi \left( x\right) \,.
\end{equation*}
\end{enumerate}

\noindent%
\textbf{Acknowledgements~~}The initial impetus to the present study is due
to Nicola Gambino. Some of the work on this paper was undertaken when Aczel
and Granstr\"{o}m visited the Mathematical Institute of the University of
Munich. Aczel is grateful to the Mathematical Institute for providing
excellent office space and other facilities, and to the Leverhulme Trust for
funding his visit with an Emeritus Fellowship. Granstr\"{o}m's visit was
enabled by a grant within the MATHLOGAPPS programme of the European Union.

\subsection*{Authors' Addresses}

Peter Aczel:\newline
School of Computer Science, University of Manchester, Manchester M13 9PL,
England; \texttt{\ petera@cs.man.ac.uk}

\medskip \noindent Benno van den Berg{:}\newline
Mathematisch Instituut, Universiteit Utrecht,\newline
P.O. Box 80010, 3508 TA Utrecht, The Netherlands; \texttt{B.vandenBerg1@uu.nl%
}

\medskip \noindent Johan Granstr{\"{o}m:}\newline
{Mathematisches Institut, Universit\"{a}t M\"{u}nchen,}\newline
{Theresienstra\ss e 39, \thinspace \thinspace 80333 M\"{u}nchen, Germany};%
\footnote{%
Current Affiliation: Google Z\"{u}rich, Brandschenkestrasse 110, 8002 Z\"{u}%
rich, Switzerland} \texttt{georg.granstrom@gmail.com}

\medskip \noindent Peter Schuster:\footnote{%
Corresponding author.}\footnote{%
Temporary address (until 30 September 2011): {Mathematisches Institut,
Universit\"{a}t M\"{u}nchen, Theresienstra\ss e 39, \thinspace \thinspace
80333 M\"{u}nchen, Germany; \texttt{pschust@math.lmu.de}}}\newline
Pure Mathematics, University of Leeds, Leeds LS2 9JT, England; {\texttt{%
pschust@maths.leeds.ac.uk}}\newline

\end{document}